\documentclass[preprint,12pt]{elsarticle}
\usepackage{booktabs}
\usepackage{amsfonts}
\usepackage{amsmath}
\usepackage{amssymb}
\usepackage{amsthm}
\usepackage{color}
\usepackage{graphicx}
\providecommand{\U}[1]{\protect\rule{.1in}{.1in}}
\newtheorem*{thm}{Theorem}
\newtheorem*{thm1}{Theorem 1}

\newtheorem*{thm2}{Theorem 2}

\newtheorem*{lem}{Lemma}

\newtheorem*{prop}{Proposition}

\newdefinition{rmk}{Remark}
\newproof{pf}{\bf{Proof}}
\newproof{pot}{Proof of Theorem \ref{thm2}}

\def\Xint#1{\mathchoice
   {\XXint\displaystyle\textstyle{#1}}%
   {\XXint\textstyle\scriptstyle{#1}}%
   {\XXint\scriptstyle\scriptscriptstyle{#1}}%
   {\XXint\scriptscriptstyle\scriptscriptstyle{#1}}%
   \!\int}
\def\XXint#1#2#3{{\setbox0=\hbox{$#1{#2#3}{\int}$}
     \vcenter{\hbox{$#2#3$}}\kern-.5\wd0}}
\def\dashint{\Xint-}
\begin{document}

\begin{frontmatter}

\title{On the corona theorem on smooth curves}
\author[esc]{J.M. Enr\'{i}quez-Salamanca\corref{cor1}}
\ead{enrique.desalamanca@uca.es}
\author[uca]{Mar\'{i}a J. Gonz\'{a}lez\corref{cor2}\fnref{fn2}}
\ead{majose.gonzalez@uca.es}
\cortext[cor2]{Corresponding author}
\cortext[cor1]{Principal corresponding author}
\address[uca]{Department of Mathematics, University of C\'{a}diz, Puerto Real, 11510, Spain}
\address[esc]{Department of Mathematics, University of C\'{a}diz, C\'{a}diz, 11002, Spain}
\fntext[cor1]{This author is supported by Grant MTM-2014-51824-P.}
\fntext[fn2]{This author is supported by Grant MTM-2014-51824-P and MTM-2017-85666-P.}

\begin{abstract}
We prove the corona theorem for domains whose boundary lies in certain smooth quasicircles. These curves, which are not necessarily Dini-smooth, are defined by quasiconformal mappings whose complex dilatation verifies certain conditions. Most importantly we do not assume any ``thickness'' condition on the boundary domain. In this sense, our results complement those obtained by Garnett and Jones (1985) and C. Moore on $C^{1+\alpha}$ curves (1987).

\end{abstract}
\end{frontmatter}
\section{Introduction}
\label{intro}
\noindent Let $\Omega\subset \mathbb{C}$ be a domain in the complex plane and let $H^{\infty}(\Omega)$ be the space of bounded analytic functions in $\Omega$.\\

\noindent The first corona theorem was proved by Carleson for simply connected domains \cite{carleson1}. Denote by $\mathbb{D}$ the open unit disk.

\begin{thm}[\bf{Carleson}]
Let $f_1(z),\ldots, f_n(z)$ be given functions in $H^{\infty}(\mathbb{D})$ and verifying that
$$|f_1(z)|+|f_2(z)|+\ldots |f_n(z)|\geq \delta >0,$$

\noindent for some $0<\delta \leq 1/2$. Suppose that $\|f_k\|_{\infty}\leq 1$, $k=1,2,\ldots,n$. Then, there exist $\{g_k\}_{k=1}^{n} \in  H^{\infty}(\mathbb{D})$ so that:
\begin{equation}
  \sum_{k=1}^{n}f_{k} g_{k}=1 \;\;\text{and}\;\; \|g_{k}\|_{\infty}\leq C(n,\delta).\nonumber
\end{equation}
\end{thm}

\noindent The functions $\{f_{k}\}$ and $\{g_{k}\}$ are called corona data and corona solutions respectively, and $\delta$ and $n$ are the corona constants.\\

\noindent In his proof, Carleson introduced what is known as Carleson measures, a fundamental tool in complex and harmonic analysis.\\

\noindent The next breakthrough for infinitely connected domains \cite{carleson2} is also due to Carleson. He used the relation between interpolating sequences, boundary thickness and the Cauchy transform to prove the corona theorem for homogeneous Denjoy domains, that is, for domains with boundary $E \subset \mathbb{R}$ (Denjoy domain) such that:
\begin{equation}
|(x-r,x+r)\cap E|>\varepsilon_0 r \;\;\text{for all} \; x\in E \;\; \text{and all} \;r>0.\nonumber
\end{equation}
\noindent Newdelman extended this result to domains $\Omega=\mathbb{C}\backslash E$ where $E$ is a homogeneous set contained in a Lipschitz graph \cite{newdelman}. The idea was to divide $\Omega$ into two overlapping simply connected regions, $\tilde{\Omega}^+$ and $\tilde{\Omega}^-$. On each region, he used Carleson's simply connected result to obtain regional corona solutions by an iterating method. On each iteration, a particular $\bar{\partial}$ equation was solved to modify $\{g_j^{\pm}\}$ so that $\max |g_j^+(z)-g_j^-(z)|$ was reduced in the overlap of the regions. See more results on corona theorem in \cite{handy} and the references within. \\

\noindent The first result for domains $\Omega=\mathbb{C}\backslash E$ that did not assume the homogeneous condition on the set $E$ was proved by Garnett and Jones on Denjoy domains, that is $E\subset \mathbb{R}$ \cite{garnettandjones}.\\
\noindent Moore \cite{moore} extended the corona theorem for domains $\Omega=\mathbb{C}\backslash E$, with $E$ lying in a $C^{1+\alpha}$ curve. For that, he first proved that Cauchy integrals on a $C^{1+\alpha}$ curve behave locally like Cauchy integrals along straight lines and then used Garnett and Jones' solutions on Denjoy domains. \\

\noindent Moore's result was proved again in \cite{nuestro} by considering quasiconformal mappings. In fact, if $f$ is a conformal mapping from the upper plane $\mathbb{R}_{+}^2$ onto the complex plane, then $\Gamma=f(\mathbb{R})$ is a $C^{1+\alpha}$ curve if and only if $f$ extends to a global quasiconformal map whose dilatation $\mu$ satisfies that $|\mu|^2/|y|^{1+\varepsilon}dxdy$ is a Carleson measure relative to $\mathbb{R}$ \cite{nuestro2}. This characterization is then used to show that $H^{\infty}(\Omega)$ functions are close to $H^{\infty}$ functions on Denjoy domains and to obtain local solutions from the Garnett and Jones' solutions.\\
\noindent Both proofs of corona theorem for $C^{1+\alpha}$ curves (\cite{moore},\cite{nuestro}) can be extended to Dini-smooth curves with slight modifications.\\

\noindent A natural question would be to extend these results, where no condition on the homogeneity of the set $E$ is required, to more general curves, such as smooth curves, that is, to Jordan curves, $\Gamma$, for which there is a parametrization $f:\mathbb{R}\rightarrow \mathbb{C}$, with $f'$ continuous and $\neq 0$.\\

\noindent This paper presents the corona theorem for domains $\Omega=\mathbb{C}\backslash E$, where $E$ lies in certain smooth curves $\Gamma$. More precisely, we will consider quasicircles which are images of $\mathbb{R}$ under a global quasiconformal mapping of the the complex plane, $\rho$, whose complex dilatation, $\mu $, has compact support and verifies one of these two conditions:
\begin{enumerate}
\item The complex dilatation $\mu$ satisfies {\bf condition 1} if:
\begin{equation}
\label{condicion1}
\int_{0}\displaystyle{\frac{\mu^{*}(t)}{|t|}}\log{\left(\displaystyle{\frac{1}{|t|}}\right)}dt<\infty,
\end{equation}
\noindent where $\mu^{*}(t)=\text{esssup}\{|\mu(z)|:0<|\text{Im}(z)|<|t|\}$ is the monotonic majorant of $\mu$.
\item The complex dilatation $\mu$ satisfies {\bf condition 2} if:
\begin{equation}
\label{condicion2}
\int_{\mathbb{R}}\displaystyle{\frac{\sigma(y)}{|y|^{3/2}}}\;dy<\infty,
\end{equation}

\noindent where $\sigma(y)$ is defined a.a. $y\in \mathbb{R}$ as $\sigma(y)=\left(\int_{\mathbb{R}}|\mu(x+\text{i}y)|^2dx\right)^{1/2} $, and there exists $C>0$ so that:
\begin{equation}
\label{condicion3}
|\mu(z_0)|\lesssim \dashint_{|z-z_0|<C|\text{Im}(z_0)|}|\mu(z)|dx\;dy,\;\; \forall z_0 \in \mathbb{C}\backslash \mathbb{R}.
\end{equation}
\end{enumerate}

\noindent We will show that, in both cases, there exists $M>0$ such that the Teichm{\"u}ller-Wittich-Belinski integral

\begin{equation}
\label{condicion10}
\int_{\mathbb{C}}\displaystyle{\frac{|\mu(z+t)|}{|z|^2}}dx\;dy <M,
\end{equation}

\noindent for every $t\in \mathbb{R}$ and, therefore, $\Gamma=\rho(\mathbb{R})$ is a smooth curve (\cite{martio}, Corollary 1.6).\\

\noindent To prove our results, we will follow a similar argument as in \cite{nuestro2}. This approach, which had been previously developed by Semmes in \cite{SEMMES}, will allow us to relate $H^{\infty}(\Omega)$ to $H^{\infty}$ on Denjoy domains.\\

\begin{thm1}
\label{theorem1}
  Let $\rho$ be a global quasiconformal mapping of the complex plane, conformal at $\infty$ and with complex dilatation $\mu$ verifying either condition 1 or condition 2. Denote $\Gamma=\rho(\mathbb{R})$. Then, given a function $f\in L^{\infty}(\Gamma)$, the Cauchy integral $C_{\Gamma}(f)\in L^{\infty}(\mathbb{C})$  if and only if $C_{\mathbb{R}}(g)\in L^{\infty}(\mathbb{C})$, where $g$ is the pullback of $f$ under the quasiconformal mapping, $g=f \circ \rho$.
\end{thm1}

\noindent We can now state our main result on corona theorem for both sets of curves. For that, we consider domains $\Omega=\mathbb{C}\backslash E$ where $E$ is a compact set with positive length contained in a quasicircle $\Gamma=\rho(\mathbb{R})$, analytic at $\infty$, such that the complex dilatation of the quasiconformal mapping $\rho$ satisfies either (\ref{condicion1}) or (\ref{condicion2}) and (\ref{condicion3}).\\

\begin{thm2}
With the notation above, let $f_1,f_2,\ldots, f_n \in H^{\infty}(\Omega)$ so that $\delta \leq \max_{k}|f_{k}(\omega)|\leq 1$ for all $\omega \in \Omega$ and some $\delta>0$. Then, there exist $g_1,g_2,\ldots,g_n \in H^{\infty}(\Omega)$ such that $f_1 g_1+f_2 g_2+\ldots+f_n g_n=1$ on $\Omega$.
\end{thm2}
\noindent Note that no condition on the homogeneity of the set $E$ is assumed in Theorem 2.\\

\noindent The paper is structured as follows. In section 2, we review some basic definitions and facts. We prove theorem 1 in section 3 and theorem 2 in section 4. Finally, in section 5, an example of this sort of smooth curves which is not Dini-smooth is presented.\\

\section{Preliminaries}
\label{preliminaries}
\noindent Let us denote complex variables by $z=x+\text{i}y$ and $\omega=u+\text{i}v$. We shall use the following notation throughout this article: $\mathbb{D}=\{z:|z|<1\}$, $|E|$ represents the Lebesgue measure of any set $E$, $\delta_{\Gamma}(\omega)$ the distance from the point $\omega$ to the curve $\Gamma$, $\text{diam}(E)$ the diameter of a set $E$ and $H^{\infty}(\Omega)$ is the space of bounded analytic functions on $\Omega$. Also, we shall write  $\bar{\partial}=\partial/\partial \bar{z}=1/2(\partial_x+i\partial_y)$ and $\partial=\partial/\partial z=1/2(\partial_x-i\partial_y)$. For a square $Q$, we will denote by $l(Q)$ its length and we will use $x\lesssim y$ as shorthand for the inequality $x \leq Cy$ for some constant $C$.\\

\noindent Given a function $f$ on a rectifiable curve $\Gamma$, define its Cauchy integral $F(z)=C_{\Gamma}(f)(z)$ off $\Gamma$ by:

\begin{equation}
F(z)=\displaystyle{\frac{1}{2\pi\text{i}}}\int_{\Gamma}\displaystyle{\frac{f(\omega)}{\omega-z}}d\omega,\;\; z \notin \Gamma. \nonumber
\end{equation}

\noindent We define the jump of $F$ across $\Gamma$ at a point $z$, $j(F)(z)$, as $F_{+}(z)-F_{-}(z)$, where $F_{\pm}$ denote the boundary values of $F$. As the classical Plemelj formula states,

\begin{equation}
F_{\pm}(z)=\pm \displaystyle{\frac{1}{2}}f(z)+\displaystyle{\frac{1}{2}} P.V.\int_{\Gamma}\displaystyle{\frac{f(\omega)}{\omega-z}}d\omega,\;\; z\in \Gamma.\nonumber
\end{equation}

\noindent Hence, $F_{+}(z)-F_{-}(z)=j(F)(z)=f(z)$.\\

\noindent Consider $\rho$ a global quasiconformal mapping of the complex plane with complex dilatation $\mu$. Thus, $\rho$ is a homeomorphism with locally integrable distributional derivatives verifying that $\bar{\partial}\rho-\mu\partial \rho=0$, $\mu\in L^{\infty}(\mathbb{C})$ and $\|\mu\|_{\infty}<1$. Suppose that $\rho$ is conformal at $\infty$, with $\rho(\mathbb{R})=\Gamma$ a rectifiable quasicircle. Let $\Omega=\mathbb{C}\backslash E$, where $E \subset \Gamma$ is a compact set with positive length. Define $\Omega_0=\rho^{-1}(\Omega)$ and $E_0=\rho^{-1}(E)$. Note that $E_0$ is a compact set of positive measure (\cite{pommerenke}, Theorem 6.8). \\

\noindent Define the space:
$$H^{\infty}(\Omega_0,\mu)=\{f\circ \rho: f\in H^{\infty}(\Omega)\}.$$

\noindent Observe that if $g \in H^{\infty}(\Omega_0,\mu)$, then $\bar{\partial} f=0$ on $\Omega$ translates into $(\bar{\partial}-\mu\partial)g=0$ on $\Omega_0$ and the jump of $g$ across $E_0$ is given by $\text{j}(g)=\text{j}(f)\circ \rho$.\\

\noindent We review some facts which follow from Semmes' approach in \cite{SEMMES}. Let $f\in H^{\infty}(\Omega)$ and $g=f\circ \rho$. Consider the jump of $g$, $\text{j}(g)$, and set $\tilde{g}=C_{\mathbb{R}}(\text{j}(g))$. If we define $H=g-\tilde{g}$, then $\bar{\partial} H=\mu \partial g$ on $\Omega_0$ and since $H$ has no jump across $E_0$, we can consider that this equation holds on all $\mathbb{C}$ in the sense of distributions. For more details see \cite{nuestro2}. We can then apply Cauchy's formula to obtain:

\begin{equation}
H(z_0)=-\displaystyle{\frac{1}{\pi}}\int_{\mathbb{C}}\displaystyle{\frac{\bar{\partial}H}{z-z_0}}dx\;dy=-\displaystyle{\frac{1}{\pi}}\int_{\mathbb{C}}\displaystyle{\frac{\mu(z)\partial g(z)}{z-z_0}}dx\;dy\;\;\; \text{for all}\;\;z_0\in \mathbb{C}.\nonumber
\end{equation}

\noindent The modulus of continuity of a function $f$ on $\mathbb{R}$ is defined by:

$$\omega_f(\delta)=\text{sup}\{|f(x_1)-f(x_2)|:x_1,x_2 \in \mathbb{R}, |x_1-x_2|\leq \delta\}.$$

\noindent The function $f$ is called Dini-continuous if

$$\int_{0}\displaystyle{\frac{\omega_f(t)}{t}}dt<\infty.$$

\noindent We say that a closed Jordan curve $\Gamma$ is Dini-smooth if it has a parametrization $f(\tau)$, $0\leq \tau \leq 2\pi$, such that $f'(\tau)$ is Dini-continuos and $\neq 0$ (see \cite{pommerenke}, section 3.3 for further results).\\

\noindent Recall that for $f:\Omega\rightarrow \Omega'$ a quasiconformal mapping between domains $\Omega$ and $\Omega'$ in $\mathbb{R}^2$, the well known Ghering's result \cite{gehring} ensures that the Jacobian $J_f$ of $f$ satisfies the reverse H{\"o}lder inequality:
\begin{equation}
\label{holder}
\left(\dashint_Q J_f^p\;dx\;dy\right)^{1/p}\leq C~\dashint_Q J_f dx\;dy,
\end{equation}
\noindent for some $p>1$, where $Q$ is a cube in $\Omega$ such that $2Q\subset \Omega$ and where $\dashint_Q$ stands for $\frac{1}{|Q|}\int_Q$.

\section{Proof of Theorem 1}

\noindent We begin this section by proving that the quasicircles defined by quasiconformal mappings with complex dilatation verifying condition 1 or 2 are, indeed, smooth curves. By (\ref{condicion10}) this result is an immediate consequence of the following proposition:\\

\begin{prop}
If $\mu$ verifies (\ref{condicion1}) or (\ref{condicion2}), then there exists $M>0$ such that for all $a\in \mathbb{R}$
\begin{equation*}
\int_{\mathbb{C}}\displaystyle{\frac{|\mu(z)|}{|z-a|}}\displaystyle{\frac{dx\;dy}{|y|}}<M.
\end{equation*}
\end{prop}
\begin{proof}
If $\mu$ verifies condition 1 and $\text{supp}(\mu)\subset B(0,R)$, where $B(0,R)$ is the ball centered at 0 and radius $R$ for some $R>0$, then for any $a\in \mathbb{R}$
\begin{eqnarray*}
& &\int_{\mathbb{C}}\displaystyle{\frac{|\mu(z)|}{|z-a|}}\displaystyle{\frac{dx\;dy}{|y|}}\lesssim \int_{-R}^{R}\displaystyle{\frac{\mu^{*}(y)}{|y|}} \left(\int_{-R}^{R}\displaystyle{\frac{1}{|x-a|+|y|}}dx\right) \;dy \lesssim \\
& &\int_{-R}^{R}\displaystyle{\frac{\mu^{*}(y)}{|y|}}\left(\int_{-R}^{R}\displaystyle{\frac{1}{|x|+|y|}}dx\right)\;dy\simeq \int_{0}^{R}\displaystyle{\frac{\mu^{*}(y)}{|y|}}\log{\left(\displaystyle{\frac{1}{|y|}}\right)}dy<\infty.
\end{eqnarray*}

\noindent Let us assume next that $\mu$ verifies (\ref{condicion2}). Then for any $a \in \mathbb{R}$:

\begin{eqnarray}
\label{condicion4}
\int_{\mathbb{C}}\displaystyle{\frac{|\mu(z)|}{|z-a|}}\displaystyle{\frac{dx\;dy}{|y|}}&=&\int_{\mathbb{R}}\displaystyle{\frac{1}{|y|}}\left(\int_{\mathbb{R}}\displaystyle{\frac{|\mu(z)|}{|z-a|}}\;dx\right)dy \nonumber \\
&\leq& \int_{\mathbb{R}}\displaystyle{\frac{1}{|y|}}\left(\int_{\mathbb{R}}\displaystyle{\frac{|\mu(z)|^2}{|y|}}dx\right)^{1/2}\left(\int_{\mathbb{R}}\displaystyle{\frac{|y|}{|z-a|^2}}\;dx \right)^{1/2}dy\\
&\simeq& \int_{\mathbb{R}}\displaystyle{\frac{\sigma(y)}{|y|^{3/2}}}dy<+\infty.
\end{eqnarray}
\end{proof}

\noindent Consider $f \in L^{\infty}(\Gamma)$  so that $C_{\Gamma}(f)\in L^{\infty}(\mathbb{C})$ and let $g=f\circ \rho$ be the pullback of $f$ under the quasiconformal mapping.\\

\noindent Following Semmes' approach as described in the previous section, set $G=C_{\Gamma}(f)\circ \rho$ and $H=G-C_{\mathbb{R}}(g)$. Since $H$ has no jump across $E_0$ and $\mu$ has compact support:

\begin{equation}
\label{condicion9}
H(z_0)=-\displaystyle{\frac{1}{\pi}}\int_{\mathbb{C}}\displaystyle{\frac{\mu(z) \partial G(z)}{z-z_0}}dx\;dy\;\;\;\text{for all} \;\;z_0 \in \mathbb{C}.
\end{equation}

\noindent We will consider the Whitney decomposition associated to $\mathbb{R}_{+}^{2}$ and $\mathbb{R}_{-}^{2}$, that is, $\mathbb{C}\backslash \mathbb{R}=\cup_{k} Q_k$, where the side length of the cube $Q_k$, $l(Q_k)$, is proportional to its distance to $\mathbb{R}$. Denote by $z_k$ the center of the cube $Q_k$.\\

\noindent {\bf Proof for Theorem 1.} Let $Q$ be a cube so that $\text{supp}(\mu)\subset Q$ and suppose that $C_{\Gamma}(f)\in L^{\infty}(\mathbb{C})$. For any $a \in \mathbb{R}$, $|z-a|\simeq |z_k-a|$ for $z\in Q_k$. Also, by the circular distortion theorem, $\delta_{\Gamma}(\rho(z))\simeq \delta_{\Gamma}(\rho(z_k))$ for $z \in Q_k$.\\

\noindent Set $G(z)=C_{\Gamma}(f)\circ \rho$. Since $\partial G(z)=C_{\Gamma}'(f)(\rho(z)) \partial \rho(z)$ and $C_{\Gamma}(f)$ satisfies that $C'_{\Gamma}(f)\lesssim C(\|f\|_{\infty})/\delta_{\Gamma}(\rho(z))$, we get by (\ref{condicion9}):

\begin{equation*}
|H(a)|\lesssim \sum_{k} \displaystyle{\frac{\mu^{*}(3y_k/2)}{|z_k-a|}}\displaystyle{\frac{1}{\delta_{\Gamma}(\rho(z_k))}}\int_{Q_k}|\partial \rho(z)|\;dx\;dy.\nonumber
\end{equation*}

\noindent On the other hand:

\begin{equation}
   \int_{Q_k} |\partial \rho(z)|\;dx\;dy\lesssim \left(\int_{Q_k}|\partial \rho(z)|^2\right)^{1/2}l(Q_k) \simeq \text{diam}(\rho(Q_k))\;l(Q_k). \nonumber
\end{equation}

\noindent Therefore, as the diameter of $\rho(Q_k)$ is comparable to $\delta_{\Gamma}(\rho(z_k))$, with comparison constants depending only on $\Gamma$:

\begin{eqnarray}
\label{condicion5}
|H(a)|&\lesssim & \sum_{k} \displaystyle{\frac{\mu^{*}(3y_k/2)}{|z_k-a|}}\displaystyle{\frac{1}{\delta_{\Gamma}(\rho(z_k))}}\text{diam}(\rho(Q_k))l(Q_k) \nonumber\\
& \simeq& \int_{Q}\displaystyle{\frac{\mu^{*}(y)}{|y|}\displaystyle{\frac{1}{|z-a|}}}\;dx\;dy\simeq \int_{-l(Q)/2}^{l(Q)/2}\displaystyle{\frac{\mu^{*}(t)}{|t|}}\log{\left(\displaystyle{\frac{1}{|t|}}\right)}dt<\infty.\nonumber
\end{eqnarray}

\noindent This proves that $H_{|\mathbb{R}} \in L^{\infty}(\mathbb{R})$ if $\mu$ verifies condition 1.\\

\noindent Consider now that $\mu$ verifies condition 2 and denote by $B_z$ the ball centered at $z$ and radius $C|y|$, where $C$ is the constant given in (\ref{condicion3}).  Then, for any $a \in \mathbb{R}$, by (\ref{condicion9}) and by (\ref{condicion3}):
\begin{equation*}
|H(a)|\lesssim \int_{\mathbb{C}}\left(\dashint_{B_z}|\mu(\omega)|du\;dv\right)\displaystyle{\frac{|\partial G(z)|}{|z-a|}}dx\;dy.
\end{equation*}
\noindent By Fubini's theorem, we get:
\begin{equation*}
|H(a)|\lesssim \int_{\mathbb{C}}\displaystyle{\frac{|\mu(\omega)|}{|w-a|}}\left(\dashint_{{Q}_{\omega}}|\partial G(z)|dx\;dy\right)du\;dv,
\end{equation*}
\noindent where $Q_\omega$ is a cube containing a ball of size comparable to $B_{\omega}$.\\

\noindent As stated before, since $|\partial G(z)|\lesssim |\partial \rho(z)|/\delta_{\Gamma}(\rho(z))$ and, by the circular distortion theorem, $\delta_{\Gamma}(\rho(\omega))\simeq \delta_{\Gamma}(\rho(z))$ for $z \in Q_{\omega}$, then
\begin{equation*}
|H(a)|\lesssim \int_{\mathbb{C}} \displaystyle{\frac{|\mu(\omega)|}{|\omega-a|\delta_{\Gamma}(\rho(\omega))}}\left(\dashint_{Q_{\omega}}|\partial \rho(z)|dx\;dy\right)du\;dv.
\end{equation*}
\noindent On the other hand:
\begin{equation}
   \dashint_{Q_{ \omega}} |\partial \rho(z)|\;dx\;dy\lesssim\left(\int_{Q_{\omega}}|\partial \rho(z)|^2\right)^{1/2}/l(Q_{\omega}) \simeq \text{diam}(\rho(Q_{\omega}))/l(Q_{\omega}). \nonumber
\end{equation}
\noindent Finally, as the diameter of $\rho(Q_{\omega})$ is comparable to its distance to $\Gamma$, with comparison constants depending only on $\Gamma$, and $l(Q_{\omega})\simeq |\text{Im}(\omega)|$, we get by (\ref{condicion4}) that:

\begin{equation*}
|H(a)|\lesssim \int_{\mathbb{C}}\displaystyle{\frac{|\mu(\omega)|}{|\omega-a|}}\displaystyle{\frac{du\;dv}{|v|}}<\infty.
\end{equation*}

\noindent This proves now proves that $H_{|\mathbb{R}} \in L^{\infty}(\mathbb{R})$ if $\mu$ verifies condition 2. \\

\noindent In both cases, as $H=G-C_{\mathbb{R}}(g)$ and $G\in L^{\infty}(\mathbb{C})$, we obtain that $C_{\mathbb{R}}(g)\in L^{\infty}(\mathbb{C})$.\\

\noindent Conversely, if $C_{\mathbb{R}}(g)$ were bounded, the same argument would show that $G$ is bounded on $\mathbb{R}$ and that $C_{\Gamma}(f)\in H^{\infty}(\mathbb{C})$.
\qed

\section{Proof of Theorem 2}

\noindent Before proceeding to the proof, let us recall the notation used in our setting. We consider a domain $\Omega=\mathbb{C}\backslash E$, where $E \subset \Gamma$ is a compact set of positive length contained in a quasicircle $\Gamma=\rho(\mathbb{R})$. We assume that the quasiconformal mapping $\rho:\mathbb{C}\rightarrow \mathbb{C}$ is conformal at $\infty$ and that $\mu_\rho$ satisfies either condition 1 or condition 2. We also define $\Omega_0=\rho^{-1}(\Omega)$, $E_0=\rho^{-1}(E)$, the space $H^{\infty}(\Omega_0,\mu)=\{f\circ \rho: f\in H^{\infty}(\Omega)\}$ and the jump of functions in $H^{\infty}(\Omega)$ and $H^{\infty}(\Omega_0,\mu)$ as in the preliminaries.\\

\noindent Next we state the following lemma that will allow us to relate corona data on $\Omega$ to corona data on the Denjoy domain $\Omega_0$.

\begin{lem}
\label{lema}
Suppose that $\text{supp}(\mu)\subset Q$ for some square $Q$ centered at a real point and that $\mu$ verifies either condition 1 or condition 2. Let $g\in H^{\infty}(\Omega_0,\mu)$ and $\tilde{g}\in H^{\infty}(\Omega_0)$ so that $\text{j}(g)=\text{j}(\tilde{g})$ and set $H=g-\tilde{g}$. Then, for all $z_0 \in \mathbb{C}$, $|H(z_0)|\leq \delta$ if $l(Q)$ is small enough.
\end{lem}
\begin{proof}
\noindent If $\mu$ verified condition 1, and from the proof of Theorem 1, we would get that for all $z_0 \in \mathbb{R}$:
\begin{equation}
|H(z_0)|\lesssim \int_{-l(Q)/2}^{l(Q)/2}\displaystyle{\frac{\mu^{*}(t)}{|t|}}\log{\left(\displaystyle{\frac{1}{|t|}}\right)}dt<\infty.\nonumber
\end{equation}

\noindent Therefore, $|H(z_0)|\leq \delta/2$ for all $z_0\in \mathbb{R}$ if $\text{supp}(\mu)$ is small enough.\\

\noindent Consider now $z_0\in \mathbb{C}\backslash \mathbb{R}$ and let $Q_0$ be the Whitney cube centered at $z_0$. Then

\begin{equation*}
|H(z_0)|\lesssim \int_{Q \backslash Q_0}\displaystyle{\frac{|\mu(z)||\partial g(z)|}{|z-z_0|}}dx\;dy+\int_{Q \cap Q_0}\displaystyle{\frac{|\mu(z)||\partial g(z)|}{|z-z_0|}}dx\;dy.
\end{equation*}

\noindent To estimate the first integral note that, for any $z\notin Q_0$, $|y_0|\lesssim |z-z_0|$ and $|z-x_0|\leq |z-z_0|+|y_0| \lesssim |z-z_0|.$ Then,
\begin{equation*}
\int_{Q \backslash Q_0}\displaystyle{\frac{|\mu(z)||\partial g(z)|}{|z-z_0|}}dx\;dy\lesssim \int_{Q \backslash Q_0}\displaystyle{\frac{|\mu(z)||\partial g(z)|}{|z-x_0|}}dx\;dy\leq \int_{-l(Q)/2}^{l(Q)/2}\displaystyle{\frac{\mu^{*}(t)}{|t|}}\log{\left(\displaystyle{\frac{1}{|t|}}\right)}dt<\delta/2.
\end{equation*}

\noindent To bound the second term consider the exponent $p>1$ in (\ref{holder}) and set $p_0=2p$. Denote $q_0=p_0/(p_0-1)$. Proceeding as in the proof of Theorem 1 we get by H{\"o}lder's inequality and (\ref{holder}) that:

\begin{eqnarray*}
& &\int_{Q\cap Q_0}\displaystyle{\frac{|\mu(z)||\partial g(z)|}{|z-z_0|}}dx\;dy \lesssim \displaystyle{\frac{\mu^{*}(3y_0/2)}{\delta_{\Gamma}(\rho(z_0))}}\int_{Q\cap Q_0} \displaystyle{\frac{|\partial \rho(z)|}{|z-z_0|}}dx\;dy \nonumber\\
&\leq& \displaystyle{\frac{\mu^{*}(3y_0/2)}{\delta_{\Gamma}(\rho(z_0))}}\left(\int_{Q\cap Q_0}|\partial \rho(z)|^{p_0}dxdy\right)^{1/p_0}\left(\int_{Q\cap Q_0}\displaystyle{\frac{1}{|z-z_0|^{q_0}}}dxdy\right)^{\frac{1}{q_0}}\nonumber\\
&\lesssim& \displaystyle{\frac{\mu^{*}(3y_0/2)}{\delta_{\Gamma}(\rho(z_0))}}|Q \cap Q_0|^{1/p_0-1/2}\left(\int_{Q\cap Q_0}|\partial \rho(z)|^{2}dx\;dy\right)^{1/2}\left(\int_{Q\cap Q_0}\displaystyle{\frac{1}{|z-z_0|^{q_0}}}dxdy\right)^{\frac{1}{q_0}} \\\nonumber
&\lesssim& \mu^{*}(3y_0/2) |Q \cap Q_0|^{1/p_0-1/2} \left(\int_{Q\cap Q_0}\displaystyle{\frac{1}{|z-z_0|^{q_0}}}dxdy\right)^{\frac{1}{q_0}}.
\end{eqnarray*}

\noindent But:

\begin{eqnarray*}
\int_{Q\cap Q_0}\displaystyle{\frac{1}{|z-z_0|^{q_0}}}dxdy&\leq&\left(\int_{Q\cap Q_0}\displaystyle{\frac{1}{|z-z_0|^{2q_0}}}dxdy\right)^{1/2}|Q\cap Q_0|^{1/2}\\
&\simeq& |y_0|^{1-q_0}|Q  \cap Q_0|^{1/2}.
\end{eqnarray*}

\noindent  Consider $2l(Q)<1/e$ so that $\log{(1/|y|)}>1$ for $z\in Q$. We then finally get that:\\

\begin{eqnarray*}
\int_{Q\cap Q_0}\displaystyle{\frac{|\mu(z)||\partial g(z)|}{|z-z_0|}}dx\;dy\lesssim \int_{-l(Q)/2}^{l(Q)/2}\displaystyle{\frac{\mu^{*}(t)}{|t|}}\log{\left(\displaystyle{\frac{1}{|t|}}\right)}dt<\delta/2.
\end{eqnarray*}

\noindent Consider now $\mu$ verifying condition 2. From the proof of Theorem 1 and (\ref{condicion4}):
\begin{equation*}
|H(z_0)|\lesssim \int_{\text{supp}(\mu)}\displaystyle{\frac{|\mu(z)|}{|z-z_0|}}\displaystyle{\frac{dxdy}{|y|}}\leq \int_{\mathbb{R}}\displaystyle{\frac{\sigma(y)}{|y|^{3/2}}}dy<\infty
\end{equation*}

\noindent for any $z_0\in \mathbb{R}$. Therefore, and as the last integral does not depend on $z_0$, $|H(z_0)|\leq \delta$ for all $z_0\in \mathbb{R}$ if supp$(\mu)$ is small enough.\\

\noindent For $z_0\in \mathbb{C}\backslash \mathbb{R}$,
\begin{equation}
\label{condicion8}
|H(z_0)|\lesssim \int_{Q \backslash Q_0}\displaystyle{\frac{|\mu(z)||\partial g(z)|}{|z-z_0|}}dx\;dy+\int_{Q \cap Q_0}\displaystyle{\frac{|\mu(z)||\partial g(z)|}{|z-z_0|}}dx\;dy=\text{I}+\text{II},
\end{equation}

\noindent where $Q_0$ is the Whitney cube centered at $z_0$.\\

\noindent Following the same argument as in the previous case and by (\ref{condicion4}):
\begin{equation*}
\text{I}=\int_{Q \backslash Q_0}\displaystyle{\frac{|\mu(z)||\partial g(z)|}{|z-z_0|}}dx\;dy\lesssim \int_{Q \backslash Q_0}\displaystyle{\frac{|\mu(z)||\partial g(z)|}{|z-x_0|}}dx\;dy\leq \int_{-l(Q)/2}^{l(Q)/2}\displaystyle{\frac{\sigma(y)}{|y|^{3/2}}}dy<\delta/2.
\end{equation*}
\\
\noindent Note now that $|\partial g(z)|\lesssim |\partial \rho(z)|/\delta_{\Gamma}(\rho(z))$ and that, by the circular distortion theorem, $\delta_{\Gamma}(\rho(z))\simeq \delta_{\Gamma}(\rho(z_0))$ for $z \in Q_0$. Then, if we apply H{\"o}lder's inequality with exponents $p_0>2$ and $q_0$ as chosen in case 1, the second integral in (\ref{condicion8}) is bounded by

\begin{eqnarray*}
&\text{II}&\leq \displaystyle{\frac{1}{\delta_{\Gamma}(\rho(z_0))}}\left(\int_{Q\cap Q_0}|\mu(z)|^{p_0}|\partial \rho(z)|^{p_0}dx\;dy\right)^{1/p_{0}}\left(\int_{Q\cap Q_0} \displaystyle{\frac{1}{|z-z_0|^{q_0}}}dx\;dy\right)^{1/q_{0}} \nonumber\\
&\simeq& |y_{0}|^{\frac{2-q_0}{q_0}}\displaystyle{\frac{1}{\delta_{\Gamma(\rho(z_0))}}}\left(\int_{Q\cap Q_0}|\mu(z)|^{p_0}|\partial \rho(z)|^{p_0}dx\;dy\right)^{1/p_{0}}.
\end{eqnarray*}

\noindent But by (\ref{condicion3}) and as the Jacobian of $\rho$ satisfies the reverse H{\"o}lder inequality  (\ref{holder}):
\begin{equation*}
\left(\int_{Q \cap Q_0}|\mu(z)|^{p_0}|\partial \rho(z)|^{p_0}dx\;dy\right)^{1/p_{0}}\lesssim \left(\dashint_{2Q_{0}}|\mu(z)|dx\;dy\right)|Q\cap Q_{0}|^{1/p_{0}-1/2}\text{diam}(\rho(Q_0)),
\end{equation*}
\noindent and:
\begin{eqnarray*}
& \text{II}& \lesssim |y_{0}|^{\frac{2-q_0}{q_0}}\cdot |y_{0}|^{\frac{2}{p_0}-1}\left(\dashint_{2Q_{0}}|\mu(z)|dx\;dy\right)=\dashint_{2Q_{0}}|\mu(z)|dx\;dy\\
&\simeq& \displaystyle{\frac{1}{|2Q_{0}|}}\int_{y_0-l(Q_0)}^{y_0+l(Q_0)}\left(\int_{x_0-l(Q_0)}^{x_0+l(Q_0)}|\mu(z)|dx\right)dy \lesssim  \displaystyle{\frac{1}{y_0^2}}\int_{y_0-l(Q_0)}^{y_0+l(Q_0)} \sigma(y)\cdot l(Q_0)^{1/2}dy \\
&\simeq& \int_{y_0-l(Q_0)}^{y_0+l(Q_0)}\sigma(y)/|y|^{3/2}dy\leq \int_{-l(Q)/2}^{l(Q)/2}\sigma(y)/|y|^{3/2}dy < \delta/2
\end{eqnarray*}
\noindent as long as $Q$ is small enough.
\end{proof}

\noindent We now prove Theorem 2 for the two settings. We will follow the same steps as in \cite{nuestro} but for the sake of completeness we will reproduce all the details.

\begin{thm2}
Let $f_1,f_2,\ldots, f_n \in H^{\infty}(\Omega)$ so that $\delta \leq \max_{k}|f_{k}(\omega)|\leq 1$ for all $\omega \in \Omega$ and some $\delta>0$. Then, there exist $g_1,g_2,\ldots,g_n \in H^{\infty}(\Omega)$ such that $f_1 g_1+f_2 g_2+\ldots+f_n g_n=1$ on $\Omega$.
\end{thm2}
\begin{proof}
\noindent Gamelin proved that it is sufficient to solve the corona problem locally \cite{Gamelin2}, i.e., that for any $\zeta \in \Gamma$ there exists a neighborhood of $\zeta$ on which it is true and such that the size of the neighborhood is determined by $\delta$, $n$ and other parameters concerning $\Gamma$ (see also \cite{garnett}, page 358).\\
\noindent  We can then assume that $\mu(z)=0$ outside a small enough square centered at a real point.\\

\noindent Let $f_{k}^{*}=f_{k}\circ \rho$ be quasiregular functions defined on $\Omega_0$. Then, the jump of $f_{k}^{*}$, $\text{j}(f_{k}^{*})$, is indeed, the pullback of $\text{j}(f_{k})$ under the mapping $\rho$, that is, $\text{j}(f_{k}^{*})=\text{j}(f_{k})\circ \rho$, where $\text{j}(f_{k})$ is the jump of $f_{k}$ across E. Note that $f_1^{*},\ldots, f_n^{*} \in H^{\infty}(\Omega_0,\mu).$\\

\noindent Set the analytic functions $\tilde{f}_{k}=C_{\mathbb{R}}(\text{j}(f_{k}^{*}))$. By theorem 1, $\tilde{f}_k\in H^{\infty}(\Omega_0)$. To show that $\{\tilde{f}_k\}$ are indeed corona data, define $H_{k}=f_{k}^{*}-\tilde{f}_{k}$ and fix $z_0 \in \Omega_0$. Then, there exists $1\leq j \leq n$ such that $\delta \leq |f_j^{*}(z_0)|\leq |H(z_0)|+|\tilde{f}_j(z_0)|$. By lemma 1, $|H(z_0)|\leq \delta/2$ if $\text{supp}(\mu)$ is sufficiently small and, therefore, $\delta/2\leq |\tilde{f}_j(z_0)|$.\\

\noindent According to Garnett and Jone's theorem for Denjoy domains \cite{garnettandjones}, there exist $\tilde{p}_1,\ldots,\tilde{p}_n\in H^{\infty}(\Omega_0)$ such that $\tilde{f}_1 \tilde{p}_1+\ldots+\tilde{f}_n \tilde{p}_n=1$ on $\Omega_0$ with $\|\tilde{p}_{k}\|_{\infty}\leq C(n,\delta)$.\\

\noindent Define $P_{k}^{*}=\text{j}(\tilde{p}_{k})$. Then, $P_{k}^{*}\in L^{\infty}(\mathbb{R})$ and $\tilde{p}_{k}=C_{\mathbb{R}}(P_{k}^{*})$. Set $P_{k}=P_{k}^{*}\circ \rho^{-1}$ on $\Gamma$ and define the bounded analytic functions $p_{k}=C_{\Gamma}(P_{k})$ on $\Omega$.\\

\noindent Although $\{p_{k}\}\in H^{\infty}(\Omega)$ with $\|p_{k}\|_{\infty}\leq C(n,\delta,\Gamma)$ by theorem 1, they are not corona solutions as they do not verify that $\sum f_{k} p_{k}=1$ on $\Omega$.\\

\noindent Consider the functions $g_{k}(\omega)=p_{k}(\omega)/(\sum_{j}f_j(\omega) p_j(\omega))$, $1\leq k \leq n$, on $\Omega$. They clearly satisfy that $\sum_{k}f_{k} g_{k}=1$. To prove they are, indeed, corona solutions, it is sufficient to show that $\sum_{j}f_{j} p_{j}$ is close to 1 and therefore bounded away from 0.\\

\noindent Let us denote $p_{k}^{*}=p_{k}\circ \rho \in H^{\infty}(\Omega_0,\mu)$. Note that, again, $\text{j}(p_{k}^{*})=\text{j}(\tilde{p}_{k})=P_{k}^{*}$. Consider $\omega \in \Omega$ and $z \in \Omega_0$ so that $\omega=\rho(z)$. Then

\begin{eqnarray}
& & |\sum_{j=1}^{n}f_{j}(\rho(z))p_{j}(\rho(z))-1|=  |\sum_{j=1}^{n}f_{j}(\rho(z))p_{j}(\rho(z))-\sum_{j=1}^{n}\tilde{f}_{j}(z)\tilde{p}_{j}(z)|\nonumber\\
&\leq& \sum_{j=1}^{n}|f_{j}^{*}(z)||p_{j}^{*}(z)-\tilde{p}_{j}(z)|+\sum_{j=1}^{n}|\tilde{p}_{j}(z)||f_{j}^{*}(z)-\tilde{f}_{j}(z)|\nonumber
\end{eqnarray}

\noindent is small enough as $f_{j}^{*}$ and $\tilde{p}_{j}$ are bounded and $|p_{j}^{*}(z)-\tilde{p}_{j}(z)|$, $|f_{j}^{*}(z)-\tilde{f}_{j}(z)|$ are also small enough due to lemma 1.
\end{proof}

\section{An example of a smooth but not Dini-smooth curve}

\noindent In this section, we provide an example of a smooth quasicircle $\Gamma=\rho(\mathbb{R})$ with $\mu_\rho$ satisfying condition 2 and such that $\Gamma$ is not a Dini-smooth curve.\\

\noindent Let $h$ be the conformal map taking $\mathbb{D}$ onto the ball $B(9/10,1/10)$, $h(z)=(9+z)/10$. Consider:

$$g(z)=2z+\displaystyle{\frac{1-z}{\log{(1-z)}}},$$

\noindent and set $f=g\circ h$. Then $f$ defines an analytic function on $\mathbb{D}$.\\

\noindent Since $f'\neq 0$ in $\mathbb{D}$, then $f$ is locally univalent. Also, $(1-|z|^2)|z f''(z)/f'(z)|\leq 1$ and Becker's univalence criteria (\cite{pommerenke}, Theorem 1.11) shows that f is indeed an univalent function.\\

\noindent A simple computation shows $\overline{\lim}_{z\rightarrow 1}(1-|z|)|f''(\bar{z})|/|f'(\bar{z})|<1$. Therefore, by Becker and Pommerenke result \cite{becker2}:

\begin{equation}
f(z)=f(1/\bar{z})+f'(1/\bar{z}) (z-1/\bar{z}),\;\; \text{for}\;\; |z|>1.\nonumber
\end{equation}

\noindent defines a quasiconformal extension of $f$ in a neighbourhood of the unit circle and

$$\left|\mu(1/\bar{z})\right|\asymp(1-|z|)\displaystyle{\frac{|f''(z)|}{|f'(z)|}}\; \; \text{for} \; z\in \mathbb{D},$$

\noindent where
\begin{equation}
\label{derivada}
f'(z)=\displaystyle{\frac{1}{10}}\left( 2-\displaystyle{\frac{1}{\log{(1-h(z))}}}+\displaystyle{\frac{1}{(\log{(1-h(z))})^2}}\right) \nonumber
\end{equation}

\noindent and
\begin{equation*}
|f''(z)|\simeq \displaystyle{\frac{1}{|\log{(1-h(z))}|^2 |1-h(z)|}}\simeq \displaystyle{\frac{1}{|\log{\frac{10}{1-z}}|^2\;|1-z|}}
\end{equation*}

\noindent in $\mathbb{D}$.\\

\noindent Consider polar coordinates $z=re^{\text{i}\theta}$. As for $r>1$, $|z-1|\simeq \theta +r-1$, then:

\begin{equation*}
|\mu(re^{\text{i}\theta})|\lesssim (r-1)\displaystyle{\frac{1}{\left(\log{\frac{10}{|1-z|}}\right)^2}}\displaystyle{\frac{1}{|1-z|}},
\end{equation*}

\noindent and a simple calculation yields $\sigma(r)\lesssim (r-1)^{1/2}/(\log{\frac{1}{r-1}})^2$, when $r \rightarrow 1^{+}.$\\

\noindent Then
\begin{equation*}
\int_{1}\displaystyle{\frac{\sigma(r)}{(r-1)^{3/2}}}dr\lesssim \int_{1}\displaystyle{\frac{1}{(r-1)(\log{\frac{1}{r-1}})^2}}dr<\infty.
\end{equation*}
\\
\noindent Since estimate (\ref{condicion3}) obviously holds for $\mu$, we have proved that $\mu$ satisfies condition 2.\\

\noindent On the other hand, the modulus of continuity of $f'$ verifies that $\omega_{f'}(t)\simeq 1/\log{(1/t)}$ and

\begin{equation}
\int_{0}\displaystyle{\frac{1}{t \log{(1/t)}}}\;dt=\infty.\nonumber
\end{equation}

\noindent Therefore, see [\cite{pommerenke2}, Theorem 3.5] the curve $\Gamma=f(\mathbb{T})$ is not Dini-smooth.

\end{document}